\documentclass{article}

\usepackage[utf8]{inputenc} 
\usepackage[T1]{fontenc} 
\usepackage{lmodern} 

\usepackage[top=3cm, includefoot]{geometry}

\usepackage{amsmath,amssymb,amsfonts,amsthm}
\usepackage{accents}
\usepackage{dsfont}

\usepackage{xifthen} 
\usepackage{graphicx} 
\usepackage{xcolor} 
\usepackage{tabularx}
\usepackage{subcaption}

\usepackage{tikz}
\usetikzlibrary{arrows.meta}
\usetikzlibrary{lindenmayersystems}
\usetikzlibrary{backgrounds}

\usetikzlibrary{decorations.pathreplacing, intersections, calc, angles, quotes, fadings, scopes, hobby, decorations.text, decorations.pathreplacing, shapes.geometric, arrows.meta, shadows}

\usepackage{pgfplots}

\newtheoremstyle{nopoint} 
  {3pt}          
  {3pt}          
  {\itshape}     
  {}             
  {\bfseries}    
  {}             
  {.5em}         
  {}             

\theoremstyle{nopoint}

\newtheorem{theorem}{Theorem}[section]
\newtheorem{definition}[theorem]{Definition}
\newtheorem{lemma}[theorem]{Lemma}

\newtheorem{remark}[theorem]{Remark}
\newtheorem{proposition}[theorem]{Proposition}

\usepackage{hyperref}
\usepackage{cleveref}

\usepackage{pdfpages}

\usepackage{todonotes}

\usepackage[all]{xy} 

\newcommand{\RR}{\mathbb{R}}
\newcommand{\NN}{\mathbb{N}}
\newcommand{\ZZ}{\mathbb{Z}}

\newcommand{\ee}{\mathrm{e}}
\newcommand{\OO}{\mathcal{O}}
\renewcommand{\gg}{\mathtt{g}}

\newcommand{\lee}{<}

\renewcommand{\epsilon}{\varepsilon}

\DeclareMathOperator{\dist}{dist}
\DeclareMathOperator{\interior}{int}
\DeclareMathOperator{\conv}{conv}
\DeclareMathOperator{\hol}{Hp}
\DeclareMathOperator{\hv}{Hv}

\counterwithin{figure}{section}

\usepackage{graphicx} 

\title{The regular n-flake dust in $\RR^2$ is not Minkowski measurable}
\author{Uta Freiberg, Jonas Lippold}
\date{\today}

\begin{document}

\maketitle

\abstract{    A long-standing conjecture of Lapidus asserts that under certain conditions  a self-similar fractal set is not Minkowski measurable if and only if it is of lattice-type. For self-similar sets in $\RR$, the Lapidus conjecture has been confirmed. However, in higher dimensions, it remains unclear whether all lattice-type self-similar sets are not Minkowski measurable. This work presents families of lattice-type subsets in $\RR^2$ that are not Minkowski measurable, hence providing further support for the conjecture.}

\section{Introduction}

The Minkowski content  can be viewed as a tool for describing the geometry of a fractal set beyond their (Hausdorff or Minkowski) dimension, particularly for distinguishing between sets of the same dimension. This makes it a natural question to determine which sets are Minkowski measurable and to understand the mechanisms that lead to measurability or its failure.

Over the past decades, substantial progress has been made in the study of Minkowski measurability, especially in the context of self-similar sets generated by iterated function systems (IFS) satisfying the open set condition (OSC). A central theme in this area is the relationship between geometric properties of the attractor and algebraic properties of the underlying system. In this direction, Lapidus formulated a conjecture in the 1990s asserting that Minkowski measurability should be equivalent to the non-lattice property of the IFS, see \cite{Lapidus}. This condition depends only on the logarithms of the contraction ratios and is therefore easy to verify, making it particularly appealing as a potential characterization. 

The non-lattice case is by now well understood: Gatzouras \cite{Gatzouras} showed that self-similar sets in $\RR^d$ satisfying the OSC are Minkowski measurable whenever the associated IFS is non-lattice. In contrast, the lattice case presents significantly greater difficulties. In dimension one, the conjecture has been completely resolved through the work of several authors, culminating in a full characterization of Minkowski measurability for self-similar subsets of $\RR$ \cite{Falconer:Minkowski, ReellerFall:neu}. However, for dimensions $d \geq 2$, the situation remains far from settled, and a complete understanding is still lacking. We refer to \cite{Kombrink:survay:on:MM} for a comprehensive overview of the current state of the art.

In higher dimensions, available results typically require additional structural assumptions. One prominent example is the pluriphase condition, which yields non-Minkowski measurability for certain lattice self-similar sets under suitable hypotheses \cite{Pluriphase}. This approach, as well as several related results, is based on renewal-theoretic methods and the analysis of scaling functions associated with the parallel volume. Another line of research focuses on exact tube formulas and employs techniques from the theory of complex dimensions and fractal zeta functions. While this framework has proven to be very powerful and has led to deep insights, it often involves technically demanding computations and is not always easily applicable to concrete families of sets.

The main objective of this paper is to investigate the Minkowski measurability of a class of self-similar sets in $\RR^2$, namely the so-called \textit{$n$-flake dusts}. These sets form families of lattice self-similar sets satisfying the open set condition, and thus provide a suitable testing ground for the conjectured relationship between lattice structure and non-measurability.

Our main result shows that, under a geometric condition on the contraction ratio, $n$-flake dusts fail to be Minkowski measurable. The proof is based on a detailed analysis of the volume of parallel sets and exploits the presence of hole points, a notion introduced in Subsection \ref{sec:hole_point}, which induce a lack of smoothness in the associated volume function. More precisely, we show that the function $\varepsilon \mapsto \lambda^2(F_\varepsilon)$ is not twice differentiable at a critical scale, which implies that the corresponding normalized volume function cannot be constant. This approach provides a new mechanism for establishing non-Minkowski measurability in the lattice case and thus contribute further evidence toward the conjectured characterization of Minkowski measurability beyond one dimension. 

The paper is organized as follows. Section~\ref{sec:preliminaries} contains the required preliminaries (see \ref{ssec: M.-mb.} to \ref{sec:hole_point}) together with a brief overview of known results (see \ref{ssec: known_results}). In Section~\ref{sec:n-flake}, we introduce the class of $n$-flake dusts and discuss their relation to the general framework. Section~\ref{sec:main-result} is devoted to the proof of the main result on Minkowski non-measurability of $n$-flake dusts.


\section{Preliminaries}
\label{sec:preliminaries}

The necessary terminology is established first to enable the formulation and proof of the main results. See \cite{Falconer:FracGeo} for background on fractal geometry. We follow the notation and terminology introduced in \cite{freiberg2026family} and add the concept of hole points.

\subsection{Minkowski measurability}
\label{ssec: M.-mb.}

Let  $(\RR^d, \| \cdot \|)$ be the $d$-dimensional Euclidean space and denote by 
$$\mathcal{H}(\RR^d) := \{ K \subset \RR^d \mid K \not= \emptyset \ \text{and} \ K \ \text{is compact} \}$$
the corresponding \emph{Hausdorff space}. For $A \in \mathcal{H}(\RR^d)$,  $\varepsilon > 0$ define the \emph{$\varepsilon$- parallel set} of $A$ as $$ A_\varepsilon := \{x \in \RR^d \mid \dist(x, A) \leq \varepsilon \},$$
where $\dist (A, B):=\inf\{ \| x-y\| \mid x\in A, y\in B \}$ denotes the distance of two non-empty sets $A,B\subset \RR^d$ and for a point $x \in \RR^d$ we write $\dist (x, B):=\dist (\{x\}, B)$. Denote by $\lambda^d$ the $d$-dimensional Lebesgue measure. If the \emph{Minkowski dimension} $$\dim_{\mathcal{M}}(A):=d-\lim_{\varepsilon \searrow 0} \frac{\ln \lambda^d(A_\varepsilon)}{\ln\varepsilon}$$ exists, then the \emph{Minkowski content} of $A$ is defined by
 $$\mathcal{M}(A):= \lim_{\varepsilon \searrow 0} \frac{\lambda^d (A_\varepsilon)}{\varepsilon^{d-\dim_{\mathcal{M}}(A)}},$$

provided this limit also exists in $[0, \infty]$. The set $A$ is called $\emph{Minkowski measurable}$ if $\mathcal{M}(A)$ exists and is both positive and finite.

\subsection{Self-similar sets, open set condition, and the (non-)lattice case}

A finite set of contractions with at least two maps is called a iterated function system (IFS). Let $S=\{ S_1, \ldots ,S_N\}$, with $N\geq 2$, denote an IFS consisting of contractive similarities $S_1, \ldots ,S_N$ acting on  $\RR^d$. Such a system is called a \emph{self-similar system}. A set $F \subset \RR^d$ is called \emph{self-similar} if $$F=\bigcup_{i=1}^n S_i(F)$$ 
for some self-similar system.

For a set $A$ denote by $\mathcal{P}(A)$ the power set of $A$. Define the to the self-similar system $S$ corresponding (set-valued) map $ \mathbf{S} : \mathcal{P}(\RR^d) \to \mathcal{P}(\RR^d) $ by
\begin{equation}
\label{eq:fett_S}
    \mathbf{S}A:=\bigcup_{i=1}^{N} S_i(A).
\end{equation}

It is well known (\cite{Hut}) that $ \mathbf{S}\vert_{\mathcal{H}(\RR^d)} : \mathcal{H}(\RR^d) \to \mathcal{H}(\RR^d) $ has a unique fixed point $F$ which is called the \emph{(self-similar) attractor} of the self-similar system $S$.

The IFS $S$ is said to satisfy the \emph{open set condition} (OSC) if there exits a non-empty bounded open set $O \subset \RR^d $ such that  
\begin{equation}
\label{eq:osc}
    \bigcup_{i=1}^N S_i(O) \subseteq O \  \ \text{and} \ \  S_i (O) \cap S_j (O) = \emptyset \ \forall \  i,j= 1, \ldots ,N \ \text{with} \  i \not = j.
\end{equation}
A self-similar system that satisfies the (OSC) is called an \emph{(OSC) self-similar system}. Any non-empty bounded open set $O$ satisfying \eqref{eq:osc} is called a \emph{feasible open set} for $S$. If $S$ satisfies the (OSC) with feasible set $O$ and in addition $O \cap F \not = \emptyset $ holds, then we say $S$ satisfies the \emph{strong open set condition} (SOSC) and $O $ is called a \emph{strong feasible open set}. Note that it was shown in \cite{Schief} that if a self-similar system satisfies (OSC), then it possesses a strong feasible open set. 

Denote the scaling ratio of $S_i$ by $r_i$, $i=1,\ldots , N$. The self-similar system $S$ is said to be \emph{lattice} if $\{ \ln r_1, \ldots , \ln r_N \}$ generates a discrete subgroup of $(\RR, +).$ Otherwise, $S$ is said to be  \emph{non-lattice}. If $S$ is lattice, then there exists a largest $a>0$ such that $\{ \ln r_1, \ldots , \ln r_N \} \subseteq a\ZZ$ and $\ee^a$ is called the \emph{base} of $S$.

\subsection{Non-trivial sets and projection condition}
Let $F \in \mathcal{H}(\RR^d)$ be the attractor of a (OSC) self-similar system $S =\{S_1, \ldots ,S_N \}$. $F$ is called \emph{non-trivial} if there exists a feasible open set $O$ such that
\begin{equation}
    O \not \subseteq \overline{\mathbf{S}O},
\end{equation}
where $\overline{B}$ denotes the closure of $B \subseteq \RR^d$; otherwise, $F$ is called \emph{trivial}. $F$ is non-trivial if and only if $F$ has empty interior, which is equivalent to $F$ having a Minkowski dimension strictly less than $d$, as shown in \cite{Pears:tilings}. In particular, non-triviality is independent of the choice of the feasible set $O$.

Let $\pi_A$ denote the \emph{metric projection} onto $A$ for a nonempty compact set $A \in \mathcal{H}(\RR^d)$, which is defined for points $x \in \RR^d$ that have a unique nearest point $y$ in $A$ by $$\pi_F(x):= y.$$ The set $O$ is said to satisfy the \emph{projection condition}, if $$S_i(O) \subseteq \overline{\pi_F^{-1}(S_i F)}$$ for all $i=1,\ldots ,N$. It is worth noting that, as long as (OSC) holds, it is always possible to find a strong feasible open set that satisfies the projection condition (see \cite{Winter:MinContent}).

\subsection{Pluriphase condition}
\label{ssec: Pl_cond}
For a given self-similar system $S$ with non-trivial attractor $F \in \mathcal{H}(\RR^d)$ and a fixed feasible open set $O$, define $$\Gamma := \Gamma_{O} := O \setminus \mathbf{S}O$$ and $$g:=\sup\limits_{x\in\Gamma} \{\min\limits_{y\in F}\{\|x-y\|\}\}.$$ The set $F$ is said to be \emph{pluriphase  with respect to} $\Gamma_{O}$ if there exists a finite partition $0=: a_0 < a_1<\ldots <a_M:= g < \infty$ of the interval $(0,\infty)$ and constants $\kappa_{m,k} \in \RR, \ m=1,\ldots ,M, \ k =0,\ldots ,d$, such that for all $\varepsilon>0$
\begin{equation}
    \lambda^{d} (F_\varepsilon \cap \Gamma)=\sum_{m=1}^M \mathds{1}_{(a_{m-1},a_m]}(\varepsilon) \sum_{k=0}^d \kappa_{m,k}\varepsilon^{d-k} + \mathds{1}_{(g,\infty)}(\varepsilon)\lambda^{d} (\Gamma).
\end{equation}
Here $\mathds{1}_A$ denotes the indicator function of a set $A$.

\subsection{Hole points}
\label{sec:hole_point}
For a set $A\in \mathcal{H}(\RR^d)$ we introduce our concept of hole points as a special case of the well known concept of critical points of the distance function $d_A(x):= \dist (x, A)$.

For \(A \subseteq \RR^d\), let \(\conv(A)\) and \(\interior(A)\) denote the convex hull and the interior of \(A\), respectively.

\begin{definition}
 Let $A \in \mathcal{H}(\mathbb{R}^d)$. The point $x \in \mathbb{R}^d \setminus A$ is called a \emph{critical point} if  $x \in \conv(\Sigma_A(x))$, where $\Sigma_A(x) := \{ a \in A : \|a - x\| = d(x, A) \}$.   
\end{definition}

\begin{remark} 
 \label{Thm:Fu}
    Let $f \colon \RR^d \to \RR$ be a locally Lipschitz function. For $x\in \RR^d$ denote $\partial^{\circ}f(x)$ the \emph{Clarke generalized gradient} of $f$ at $x$ (also called the \emph{Clarke subdifferential}). Suppose that $p\in \RR^d$ has $0\in \partial^\circ f(p)$, then $p$ is called \emph{critical point of} $f$ (cf. \cite{Clarke1}, Chap. 2, and \cite{Clarke2}).

    According to a result by Fu (\cite{Fu}, Lemma~4.2), for any $A \in \mathcal{H}(\mathbb{R}^d)$, a point $x \in \mathbb{R}^d \setminus A$ is a critical point of the distance function $d_A$ (in the sense of Clarke) if and only if $x$ is a critical point. 
    
\end{remark}

\begin{definition}
   Let $A\in \mathcal{H}(\RR^d)$. A point $x \in \RR^d$ is called \emph{hole point} of $A$ if  $x \in \ \interior(\conv( \Sigma_A(x)))$. Denote by $\hol (A)$ the set of all \emph{hole points} of $A$. The set  $\hv (A) := \{ d(x,A) : x \in \hol (A) \}$ denotes the set of \emph{hole values}.
\end{definition}

\begin{definition}
     Let $A\in \mathcal{H}(\RR^d)$ and let $\epsilon > 0$. A set $H \subset \RR^d$ is called a \emph{hole} of $A_\epsilon$ if $H$ is a bounded connected component of  $\mathbb{R}^d \setminus A_\epsilon$.
\end{definition}

\begin{lemma}
Let $A \in \mathcal{H}(\mathbb{R}^d)$. Then every $x\in\hol(A)$ lies in a hole of $A_\epsilon$ for some $\epsilon>0$.
\end{lemma}

\begin{proof}

 Let $x \in \hol (A)$ and $v_x := d(x, A)$. By definition of $\Sigma_A(x)$, for all $a \in \Sigma_A(x)$, $\|a - x\| = v_x$. By Caratheodory's Theorem, there exists a finite subset $S = \{a_0, \dots, a_k\} \subseteq \Sigma_A(x)$ (with $k \leq d$) such that $x \in \text{int}(\text{conv}(S))$. Let $P := \text{conv}(S)$ be this simplex. 

For each $(d-1)$-dimensional (maximal) facet $\Phi \subset \partial P$, let $h_{\Phi}$ be the distance from $x$ to (the affine hull of) $\Phi$. Since $x \in \text{int}(P)$, we have $h_{\Phi} > 0$ for all facets. The circumradius $R_{\Phi}$ of each facet $\Phi$ is given by $R_{\Phi} = \sqrt{v_x^2 - h_{\Phi}^2} < v_x$.

Let $R_{max} = \max_{\Phi} R_{\Phi}$. Choose $\epsilon$ such that $R_{max} < \epsilon < v_x$. 
First, since $\epsilon < v_x = d(x,A)$, it holds that $x \notin A_\epsilon$. 
Second, consider any path $\gamma(t) = x + tv$ for a unit vector $v \in \mathbb{R}^d$. Since $x \in \text{int}(P)$, every ray must intersect some facet $\Phi$ of $\partial P$ at a point $p$. The distance from $p$ to the nearest vertex $a_i$ of $\Phi$ is at most $R_{\Phi} \leq R_{max} < \epsilon$. Thus, $p \in B(a_i, \epsilon) \subseteq A_\epsilon$. 

Because every path from $x$ to infinity must cross $\partial P \subset A_\epsilon$, $x$ is contained in a bounded connected component of $\mathbb{R}^d \setminus A_\epsilon$.
\end{proof}

\begin{remark}
    The sets $\hol (A)$ and $\hv (A)$ are not necessarily closed. Example:  $n=4$ and $r<1/2$ it is easy to see $\hol (\mathcal{F}^{n,r})$ and $\hv (\mathcal{F}^{n,r})$ are not closed; see Section \ref{sec:n-flake} for the definition of $\mathcal{F}^{n,r}$.
\end{remark}

\subsection{Known results on Minkowski measurability of self-similar sets}
\label{ssec: known_results}
The principal results on Minkowski measurability for self-similar sets in $\RR^d$ are reviewed in order to clarify its relationship with the non-lattice property. A schematic illustration is provided in Figure~\ref{fig:Schaubild}.
\begin{theorem}[Gatzouras \cite{Gatzouras}] \label{thm: nonlattice_folgt_M-mb}
Let $F \in \mathcal{H}(\RR^d)$ be the attractor of a non-lattice (OSC) self-similar system. Then $F$ is  Minkowski measurable.
\end{theorem}

\begin{theorem}
    [Falconer \cite{Falconer:Minkowski} and Kombrink, Winter \cite{ReellerFall:neu}] \label{thm: R^1_lattice_folgt_nicht-M-mb} Let $F \in \mathcal{H}(\RR)$ be the attractor of a lattice (OSC) self-similar system with $\dim_{\mathcal{M}}F \lee 1$. Then $F$ is not Minkowski measurable.
\end{theorem}

Taken together, Theorems \ref{thm: nonlattice_folgt_M-mb} and \ref{thm: R^1_lattice_folgt_nicht-M-mb} provide a full characterization of Minkowski measurability for self-similar sets in $\RR$ generated by an (OSC) IFS $S$. If the Minkowski dimension is below $1$, Minkowski measurability holds if and only if $S$ satisfies the non-lattice condition. It is currently unknown whether this characterization extends to higher-dimensional settings. A partial step toward this question is given in Theorem \ref{thm: pluriphase-case}.

\begin{theorem}[Kombrink, Pearse, Winter, \cite{Pluriphase}] \label{thm: pluriphase-case}
Let $F \in \mathcal{H}(\RR^d)$ be the attractor of an lattice (OSC) self-similar system with $\dim_{\mathcal{M}} F \not \in \NN$. Suppose that there exists a strong feasible open set $O$ satisfying the projection condition such that $F$ is pluriphase with respect to $\Gamma_{O}$. Then $F$ is not Minkowski measurable.
\end{theorem}

\begin{remark}
Note that it is essential to exclude sets with an integer Minkowski dimension from the Lapidus conjecture. This limitation can be illustrated by a continuous family of one-dimensional counterexamples. For any parameter $r \in (0,1)$, consider the self-similar system $S^{(r)}$ on $\mathbb{R}$ defined by
\[
S_1^{(r)}(x) = r x, \quad S_2^{(r)}(x) = (1-r)x + r.
\]
The attractor is invariably the unit interval $F = [0,1]$ and trivially Minkowski measurable (and pluriphase) for all $r$. However, the system is lattice if and only if $\log(r)/\log(1-r) \in \mathbb{Q}$ (e.g., when $r=1/2$). For all other choices (e.g., $r=1/3$), the system is non-lattice.
Similarly, a simple two-dimensional counterexample is given by the (OSC) self-similar system $S$ on $\mathbb{R}^2$ defined by
\[
S_1(x) = \tfrac12\,x, \quad
S_2(x) = \tfrac12\,x + (\tfrac12,0), \quad
S_3(x) = \tfrac12\,x + (0,\tfrac12), \quad
S_4(x) = \tfrac12\,x + (\tfrac12,\tfrac12).
\]
The attractor of $S$ is the unit square $F=[0,1]^2$. Clearly, $F$ is Minkowski measurable (and pluriphase), but the underlying IFS $S$ is lattice.
\end{remark}

\begin{figure}

    \centering
    \begin{tikzpicture}[%
    block/.style={
      draw,
      rectangle,
      rounded corners=5pt,
      fill=orange!20,
      minimum width=2.5cm,
      minimum height=4cm,
      align=center,
      drop shadow={shadow xshift=1pt,shadow yshift=-1pt,opacity=0.4}
    },
    block2/.style={
      draw,
      rectangle,
      rounded corners=5pt,
      fill=teal!20,
      minimum width=2.5cm,
      minimum height=4cm,
      align=center,
      drop shadow={shadow xshift=1pt,shadow yshift=-1pt,opacity=0.4}
    },
    arrow/.style={
      double distance=2pt,
      -{Latex[length=3mm]},
      thick,
      draw=black!50!green,
      shorten >=3pt,
      shorten <=3pt
    },
     arrow2/.style={
      double distance=2pt,
      -{Latex[length=3mm]},
      thick,
      draw=black!20!blue,
      shorten >=3pt,
      shorten <=3pt
    }
  ]

  \node[block]  (A) at (-6,0) {non-\\lattice};
  \node[block] (B) at (6,0)  {Minkowski\\measurable};

  \draw[arrow]
    ([yshift=45pt]A.east) -- ([yshift=45pt]B.west)
    node[midway, above, font=\small\itshape] {$\mathbb{R}^{\ge1}$}
    node[midway, below, font=\small\itshape] {\emph{Gatzouras 2000}};

  \draw[arrow2]
    ([yshift=0pt]B.west) -- ([yshift=0pt]A.east)
    node[midway, above, font=\small\itshape] {$\mathbb{R}^1,\ \dim_{\mathcal M}<1$}
    node[midway, below, font=\small\itshape] {\emph{Lapidus 1993, Falconer 1995, Kombrink and Winter 2020}};

  \draw[arrow2]
    ([yshift=-45pt]B.west) -- ([yshift=-45pt]A.east)
    node[midway, above, font=\small\itshape] {$\mathbb{R}^{\ge2},\,\dim_{\mathcal M}\notin\mathbb N$, \emph{pluriphase\,w.r.t.}\,\(\mathcal O\)}
    node[midway, below, font=\small\itshape] {\emph{Kombrink, Pearse and Winter 2016}};

\end{tikzpicture}
    \vspace{0.3cm}
    \caption{Overview of the relationships between non-lattice and Minkowski measurability of self-similar sets satisfying the (OSC), see Figure 1 in \cite{freiberg2026family}.} 
    \label{fig:Schaubild}
\end{figure}
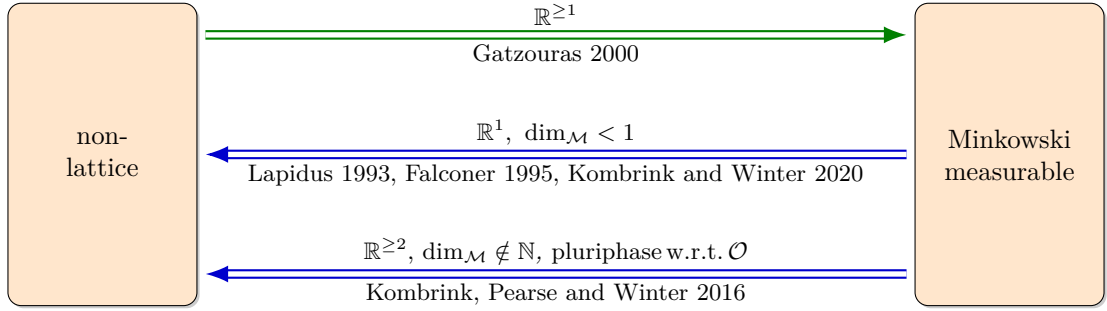

\section{n-flake dusts}
\label{sec:n-flake}

Let $n \in \NN$ with $n \geq 3$ and $r \in (0,1)$. For a regular $n$-gon with vertices $V_1, \dots, V_n \in \mathbb{R}^2$, 
an IFS is given by the set of $n$ contractive similarity mappings $\mathcal{S}:=\mathcal{S}^{n,r}:=\{\mathcal{S}_1, \dots, \mathcal{S}_n\}$: \[
\mathcal{S}_i(x):= r x + (1 - r)V_i, \quad \text{for } i =  1, \dots, n.
\]

Note, the contraction ratio $ r$ required for the sub-polygons to 
just touch is:

\begin{equation}
\label{eq:scaling_ratio}
   r = \frac{1}{2} \left(1+\displaystyle \sum_{k=1}^{\lfloor n/4 \rfloor} \cos\frac{2\pi k}{n} \right)^{-1} =: R_n, 
\end{equation}

see for example \cite{SchlickerDennisNGons}. The attractor of $\mathcal{S}$ is called \emph{$n$-flake} or \emph{Sierpinski $n$-gon}. If $r < R_n$, then the attractor of $\mathcal{S}$ is specified as \emph{$n$-flake dust} and denoted by $\mathcal{F}:=\mathcal{F}^{n,r}$. 

Since $r < R_n$, the IFS $\mathcal{S}$ satisfies (OSC), with the feasible open set $$\OO := \OO_n := \interior (\conv (\{V_1, \ldots ,V_n \})). $$  Indeed, even the
strong separation condition, i.e
\(S_i(\mathcal{F})\cap S_j(\mathcal{F})=\emptyset\) for all \(i\neq j\) holds. Therefore, by the Hutchinson dimension formula \cite{Hut}, $\mathcal{F}^{n,r}$ has Minkowski (and Hausdorff) dimension $$D:=D^{n,r}:=\dim_{\mathcal{M}}(\mathcal{F})= - \frac{\ln n}{\ln r} \in (0,2).$$ Thus $\mathcal{F}^{n,r}$ is nontrivial for all $n \geq 3$ and all $r \in (0,R_n)$. Further, $\mathcal{F}^{n,r}$ is lattice with base $r$ for all $n \geq 3$ and all $r \in (0,R_n)$. 

\begin{figure}[htbp]
    \centering
    \begin{subfigure}[b]{0.23\linewidth}
        \centering
        \includegraphics[width=\linewidth]{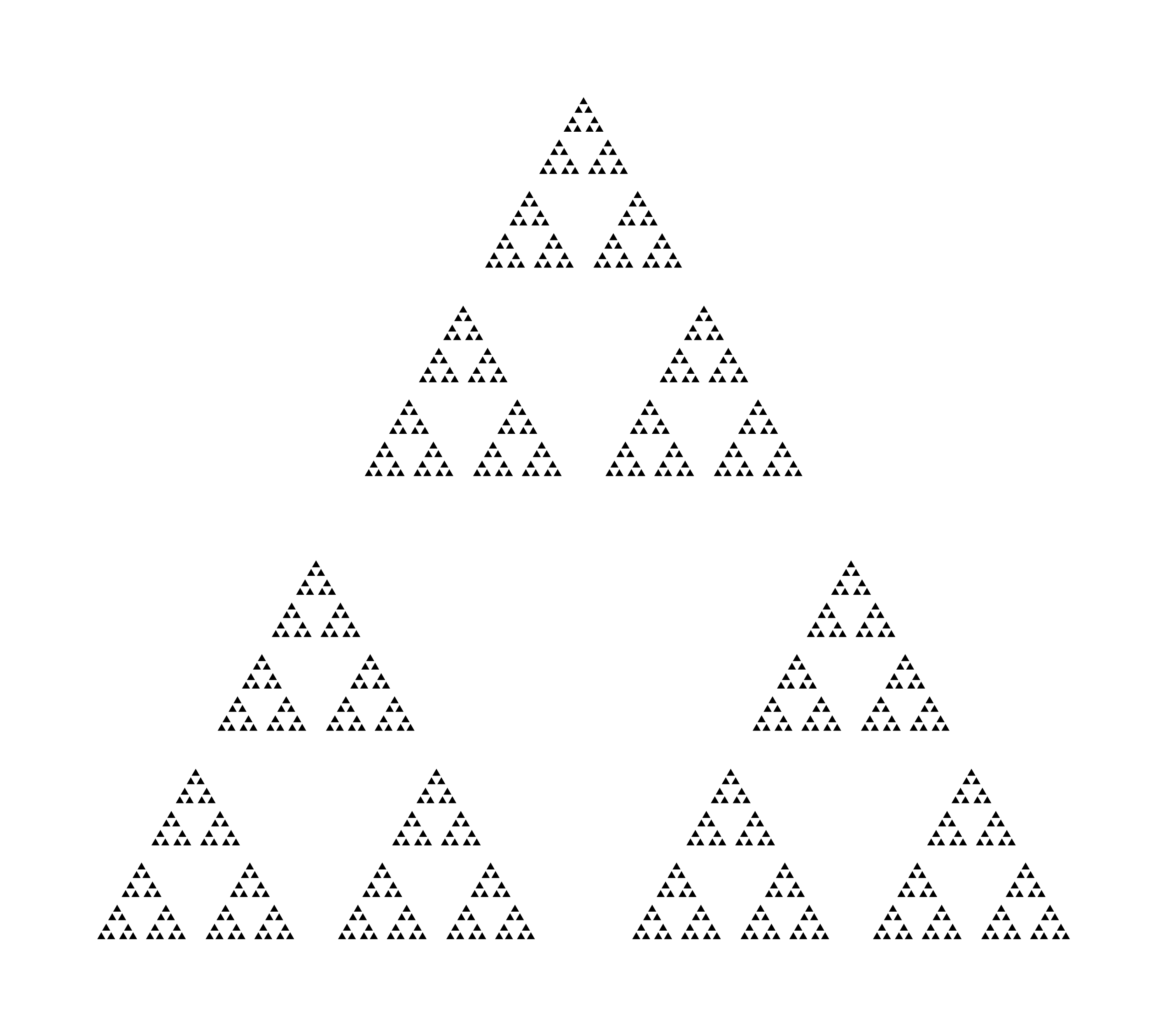}
        \caption{$n=3, r=0.45$}
        \label{fig:bild1}
    \end{subfigure}
    \hfill
    \begin{subfigure}[b]{0.23\linewidth}
        \centering
        \includegraphics[width=\linewidth]{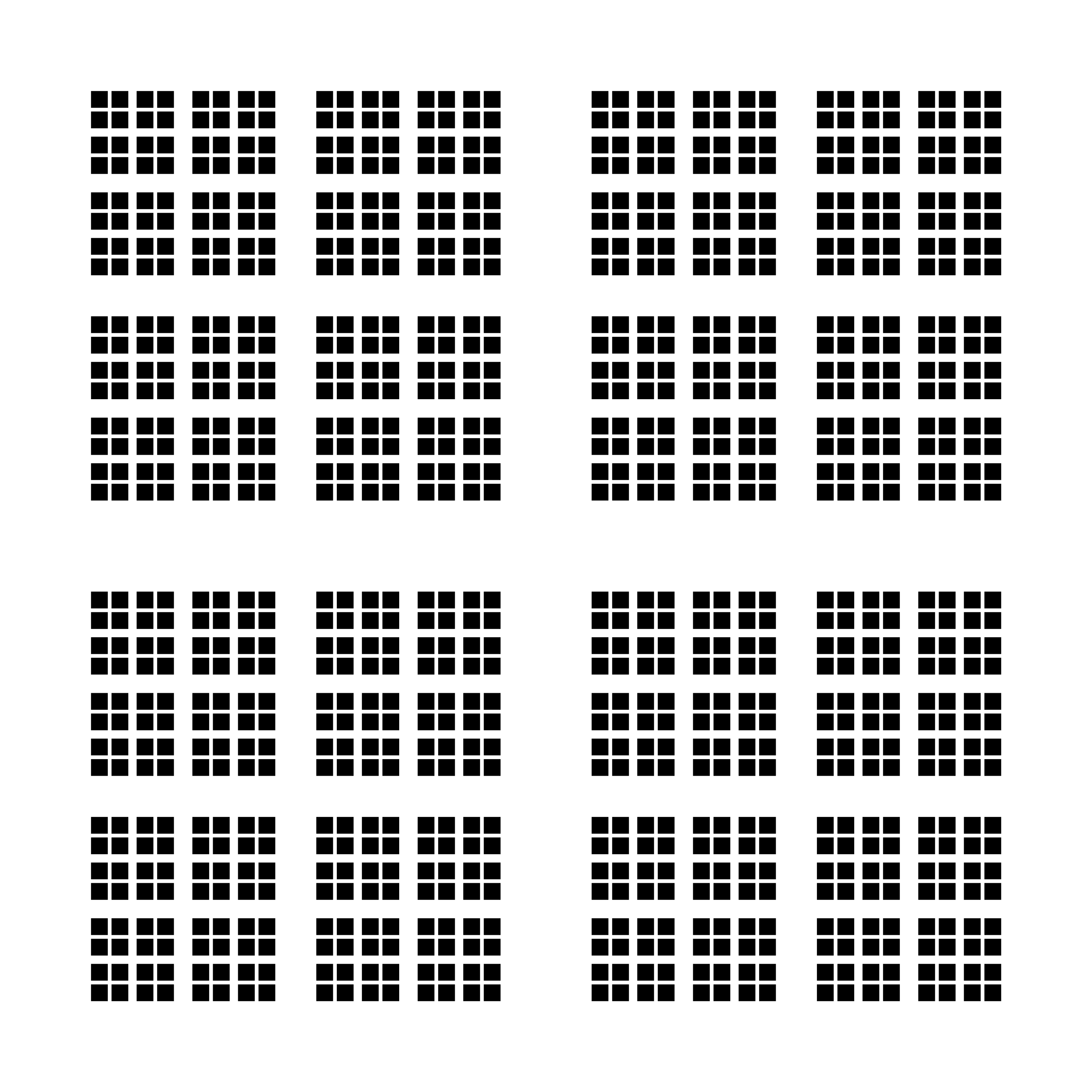}
        \caption{$n=4, r=0.45$}
        \label{fig:bild2}
    \end{subfigure}
    \hfill
    \begin{subfigure}[b]{0.23\linewidth}
        \centering
        \includegraphics[width=\linewidth]{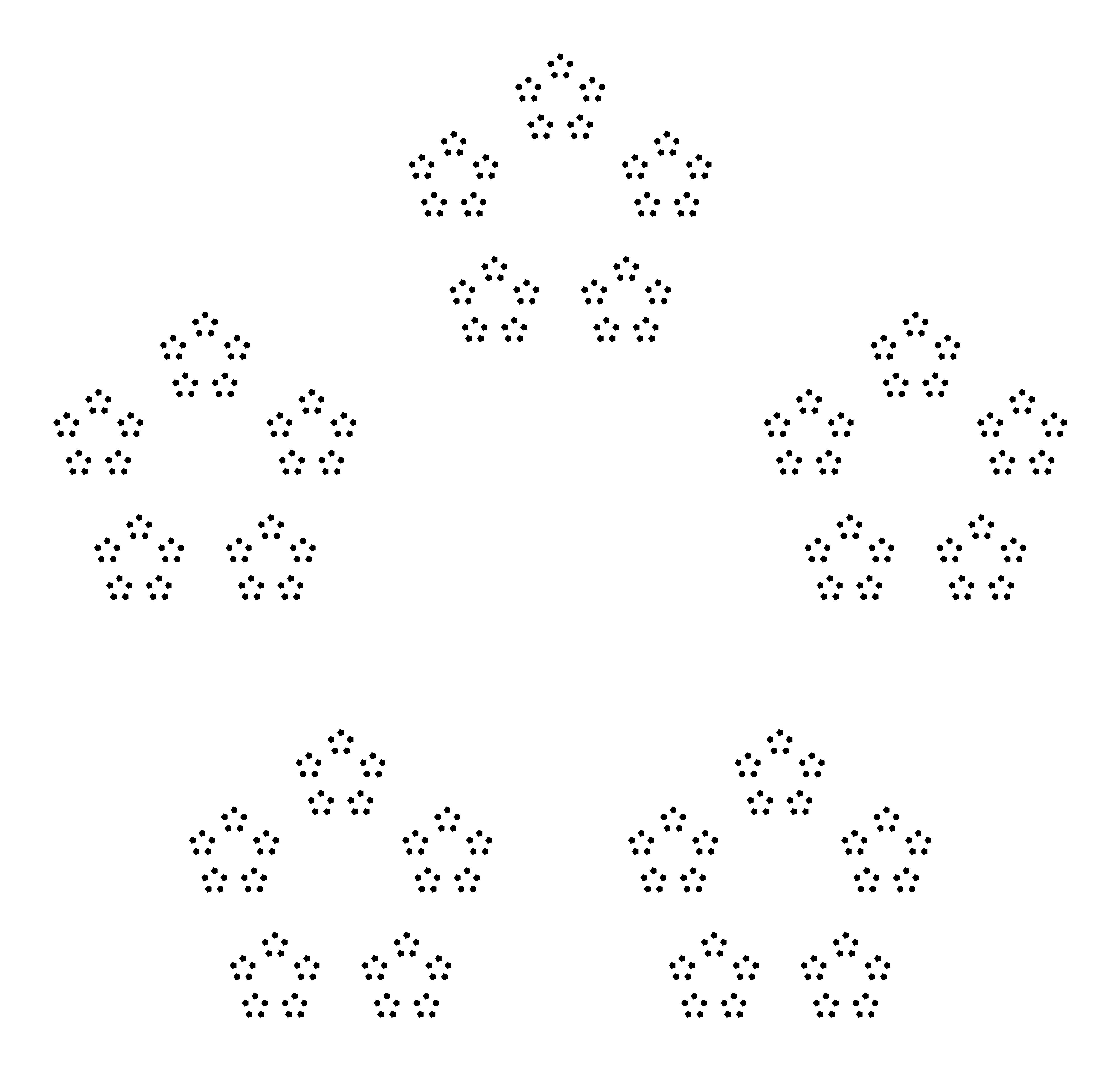}
        \caption{$n=5, r=0.3$}
        \label{fig:bild3}
    \end{subfigure}
    \hfill
    \begin{subfigure}[b]{0.23\linewidth}
        \centering
        \includegraphics[width=\linewidth]{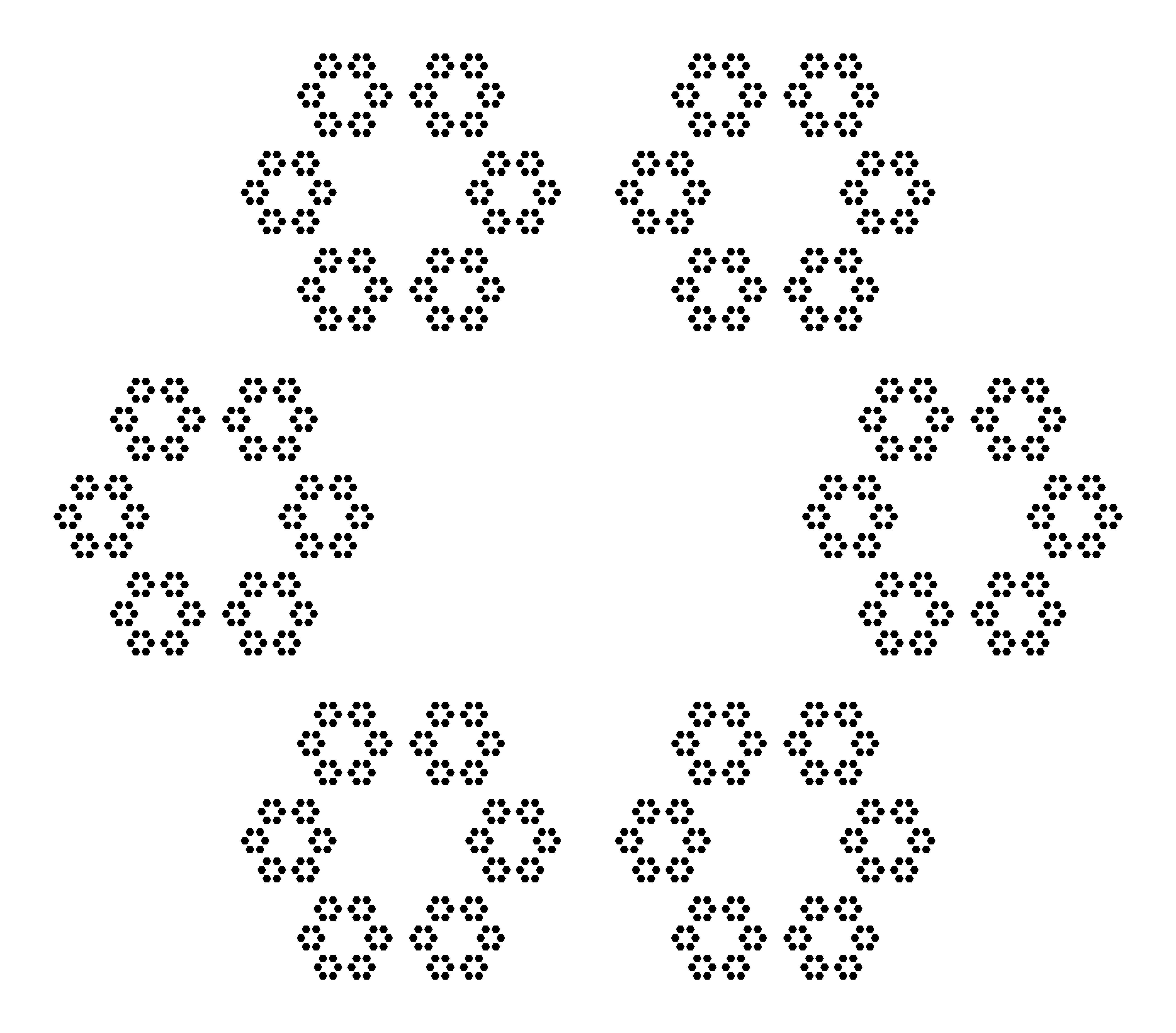} 
        \caption{$n=6, r=0.3$}
        \label{fig:bild4}
    \end{subfigure}

   \caption{Visualization of the families $\mathcal{F}^{n,r}$ for various values of $n$ and $r$.}
    \label{fig:vier_nebeneinander}
\end{figure}

\begin{remark}
For $n=4$ and $r<R_4=1/2$ we have $\mathcal{F}^{n,r}= C^{1/r}$ the Cantor dust set discussed in \cite{freiberg2026family}. Therein,  it is shown that $\mathcal{F}^{4,r}$ is not Minkowski measurable for $r< 1/30.$ For $1/2 > r > 1/30$ it is only conjectured that $C^{1/r}$ is not Minkowski measurable. In \cite{freiberg2026family} it is further shown that, for all $r \in (0, 1/2)$ $\mathcal{F}^{4,r}$ is  not pluriphase  with respect to $\Gamma_{\OO}$ for the IFS $\mathcal{S}^{4,r}$ and the feasible set $\OO =\mathcal{O}_4$. Consequently, Theorem \ref{thm: pluriphase-case} is not readily applicable to proof that $\mathcal{F}^{n,r}$ is not Minkowski measurable. 
\end{remark}

\section{Main Result}
\label{sec:main-result}

In this section we study the Minkowski measurability of the sets $\mathcal{F}=\mathcal{F}^{n,r}$ for any $n\in \NN$. The main result is stated in Theorem~\ref{thm:mainThm}: for each $n \geq 3$, a family of $n$-flake dusts $\mathcal{F}$ is not Minkowski measurable.

Let $a > 0$. A function $h: (0, \infty) \to \mathbb{R}$ is said to be \emph{log-periodic} on $(0,a]$ with scaling factor $\rho \in (0,1)$, if it satisfies the equation $$h(\rho x) = h(x) \quad \forall x \in (0, a].$$ For $\mathcal{F}$ consider the function $f \colon (0, \infty) \to [0, \infty )$
\[
f(\varepsilon) := \frac{\lambda^2(\mathcal{F}_\varepsilon)}{\varepsilon^{2-D}} .
\]

If $f$ is log-periodic on $(0,a]$ for $a>0$, then $\mathcal{F}$ is Minkowski measurable if and only if $f$ is constant. Thus, to show that $\mathcal{F}$ is not Minkowski measurable it suffices to prove that 
\begin{enumerate}
    \item[1)] $f$ is log-periodic for sufficiently small $\varepsilon$ and 
    \item[2)] $f$ is not constant.
\end{enumerate}  
  For the latter, it is enough to show that $f$ is not $C^2$ at some point $\epsilon_0 < a$. By the following Lemma~\ref{lemma:quotient_ist_nicht_konst}, this reduces to showing that the function $\varepsilon \mapsto \lambda^2(\mathcal{F}_\varepsilon)$ is not $C^2$ at $\epsilon_0$, where  $C^k$ ($k\in \NN$) denotes the $k$-times continuously differentiable functions.

\begin{lemma} \label{lemma:quotient_ist_nicht_konst}
    Let $u$ be a $C^k$ function that is not $C^{k+1}$ for some $k \in \NN$, and let $v$ be a $C^{k+1}$ function that is nowhere zero. Then their quotient $u/v$ cannot be constant.
\end{lemma}

\begin{proof}
    Suppose that the quotient is constant, meaning $u/v = c$ for some constant $c \in \RR$. This implies that $u = c \cdot v$. Since $v \in C^{k+1}$ it follows that $u \in C^{k+1}$. This directly contradicts the assumption that $u$ is not $C^{k+1}$.
\end{proof}



To proof that the function $\varepsilon \mapsto \lambda^2(F_\varepsilon)$ is not $C^2$ we recall the following well known result:

\begin{theorem}[{\cite[Corollary 2.5]{WinterRataiSurfaceArea}}, based on \cite{stacho} and \cite{HugLastWeil}] \label{Thm: Stracho}
   Let $A\in \mathcal{H}(\RR^d)$ and $\epsilon \mapsto V_A(\epsilon):=\lambda^d(A_\epsilon)$ differentiable in $\epsilon_0$. Then $V_A^\prime(\epsilon_0)= \mathcal{H}^{d-1}(\partial A_{\epsilon_0})$.
\end{theorem}

Intuitively, the key observation is that the function $\varepsilon \mapsto \lambda^2(\mathcal{F}_\varepsilon)$ is no longer $\mathcal{C}^2$ at the value $\epsilon_0 := \sup \{ \epsilon > 0 : \mathcal{F}_\varepsilon \ \text{has a hole}  \}$. Figure \ref{fig:Pi_c} and Figure \ref{fig:hole} illustrate the appearance of this "first hole" depending on $n$, $r$ and $\epsilon$. 

Next, we develop a lemma showing that the area of the "first hole" is not $C^2$, which is a crucial reason why the map $\varepsilon \mapsto \lambda^2(\mathcal{F}_\varepsilon)$ is not $C^2$.

Let $\gg>0$. Let $P_i \in \RR^2,$ $i=1, \ldots, n$ the vertices of a regular $n$-gon with  outer radius $\gg$. Denote its centroid by $O$. Let $\theta \in [0, \frac{\pi}{n})$ and denote $Q_i \in \RR^2,$ $i=1, \ldots, n$ the vertices of the $n$-gon rotated counter-clockwise by $\theta$, see Figure \ref{fig:hole}. Let $\epsilon >  \frac{1}{2}| \overline{Q_1 P_n} |$, where $\overline{PQ} :=\conv \{P , Q \}$ denotes the line segment between $P,Q \in \RR^2$ with length $| \overline{PQ}| := \dist (P,Q)$. For $p \in \RR^2$ and $\delta \in \RR$ denote by $B(p, \delta):= \{ x\in \RR^2 \mid \| x -p \| \leq \delta \} $ the disc with center $p$ and radius $\delta.$ Denote by  
$$\mathbb{B}(\epsilon):=\mathbb{B}_{n,g,\theta} (\epsilon):= B(O,\gg) \setminus \bigcup_{i=1}^n \Big( B(P_i, \epsilon) \cup B(Q_i, \epsilon) \Big)$$ 
the blue set marked in Figure \ref{fig:hole}. Observe that $\mathcal{F}_\epsilon$ depends only on $n$, $r$ and $\epsilon$. The fixed parameters $g$ and $\theta$ are introduced to facilitate the investigation of the area of $\mathbb{B}$ as a function of $\epsilon$.

For a set $A \subset \mathbb{R}^2$, let $\partial A$ denote its boundary. Furthermore, we define the points $E_1$ and $E_2$ via the intersections:
\begin{align*}
    \{E_1\} &:= \partial B(Q_1, \epsilon) \cap \partial B(P_n, \epsilon) \cap B(O, \gg), \\
    \{E_2\} &:= \partial B(Q_n, \epsilon) \cap \partial B(P_n, \epsilon) \cap B(O, \gg).
\end{align*}

Let $A_T(\epsilon)$ be the area of the triangle $E_1E_2O$ (see Figure \ref{fig:triangle}) and let $A_S(\epsilon)$  be the area of the circular segment of $B(P_n, \epsilon)$ with secant $\overline{E_1E_2}$, determined by radius $\epsilon$ and angle $\gamma (\epsilon) := \angle E_1P_nE_2$ (see Figure \ref{fig:segment}). The area of $\mathbb{B}$ as a function of $\epsilon$ is then given by:

    $$H(\epsilon):= \lambda^2(\mathbb{B}(\epsilon)) = \begin{cases}
2n[ A_T(\epsilon) -  A_S(\epsilon)], &  \frac{1}{2}| \overline{Q_1 P_n} | < \epsilon  < \gg \\
    0 & \epsilon \geq \gg .
 \end{cases} $$

Note that, in the proof of Theorem \ref{thm:mainThm}, Steps 2–4 establish that the "first hole" of  $\mathcal{F}_\epsilon$ corresponds to $\mathbb{B}_{n,g,\theta}$ for a specific range of $\epsilon$ and fixed parameters $n$, $g$ and $\theta$.

\begin{lemma}
\label{lemma:holefunction}

    The second left derivative of $H(\epsilon)$ at $\epsilon=\gg$ is  non-zero.
\end{lemma}

\begin{figure}

    \centering
    \usetikzlibrary{decorations.pathreplacing, intersections, calc, angles, quotes}

 \begin{tikzpicture}[scale = 0.25]
  \def\g{9cm}
  \def\rsmall{4.5cm} 
  \coordinate (O) at (0, 0);

  \foreach \k in {1,...,5} {
    \pgfmathsetmacro{\pangle}{90 + 72*(\k-1)}
    \coordinate (P\k) at (\pangle:\g);
    \pgfmathsetmacro{\qangle}{75 + 72*(\k-1)}
    \coordinate (Q\k) at (\qangle:\g);
  }

  \fill[blue!20] (O) circle (\g);

  \foreach \k in {1,...,5} {
    \fill[white] (P\k) circle (\rsmall);
    \fill[white] (Q\k) circle (\rsmall);
  }

  \draw[dotted] (O) circle (\g);

  \foreach \k in {1,...,5} {
    \draw[ultra thin] (P\k) circle (\rsmall);
    \draw[ultra thin] (Q\k) circle (\rsmall);

    \fill (P\k) circle (5pt);
    \fill (Q\k) circle (5pt);
    
    \pgfmathsetmacro{\pangle}{90 + 72*(\k-1)}
    \node[font=\scriptsize] at (\pangle:\g + 1.5cm) {$P_{\k}$};
    \pgfmathsetmacro{\qangle}{75 + 72*(\k-1)}
    \node[font=\scriptsize] at (\qangle:\g + 1.5cm) {$Q_{\k}$};
  }

  \fill (0,0) circle (5pt);
  \node[below left, font=\scriptsize] at (0,0) {$O$};

    \draw[thin, dotted] (0,0) -- (P1) node[ pos=0.33, left] {$g$};
    \draw[thin, dotted] (Q1) -- (P5);
    \draw [decorate, decoration={brace, amplitude=7mm, raise=2pt}]
      (Q1) -- (P5) node [midway, sloped, above=7mm, font=\small] {$2l_2$};
    \draw[thin, dotted] (P1) -- ++(45:45mm) node [midway, above ] {$\epsilon$};

\end{tikzpicture}
    \vspace{0.3cm}
    \caption{  Illustration of the set $\mathbb{B}_{n,g,\theta}(\epsilon)$ (marked in blue) for $n = 5$ and $\theta = \pi/12$. 
    } 
    \label{fig:hole}
\end{figure}
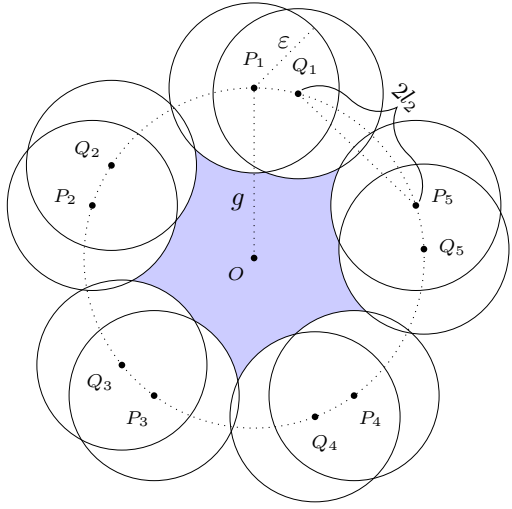 

\begin{proof}

    Let $n\in \NN$, $n \geq 3$, $\theta \in [0,\frac{\pi}{n})$. As depicted in Figure \ref{fig:triangle}, denote
    $$l_1 := \frac{1}{2} |\overline{P_n Q_n}|= \gg \sin\left(\frac{\theta}{2}\right) \ \ \text{and} \ \ l_2 := \frac{1}{2}|\overline{Q_1 P_n}| = \gg \sin\left(\frac{\pi}{n}-\frac{\theta}{2}\right),$$ and further $$h_1:=\dist(O, \overline{P_nQ_n})=\gg\cos \left(\frac{\theta}{2}\right)\ \ \text{and} \ \ h_2 :=\dist (O, \overline{Q_1P_n})= \gg\cos \left(\frac{\pi}{n}-\frac{\theta}{2}\right).$$
    Let $\epsilon \in [l_2, \infty).$ Then the area of the red triangle in in Figure \ref{fig:triangle} is  $$A_T(\epsilon)= \frac{a(\epsilon) \cdot b(\epsilon) \cdot \sin\left(\frac{\pi}{n}\right)}{2},$$ where $a(\epsilon)$ and $b(\epsilon)$ are given by
    $$a(\epsilon)=h_2 - \sqrt{\epsilon^2 -l_2^2}=\gg \cos\left(\frac{\pi}{n}-\frac{\theta}{2}\right)-\sqrt{\epsilon^2 - g^2 \sin^2\left(\frac{\pi}{n}-\frac{\theta}{2}\right) }$$
    $$b(\epsilon)=h_1 - \sqrt{\epsilon^2 -l_1^2}=\gg \cos\left(\frac{\theta}{2}\right)-\sqrt{\epsilon^2 - \gg^2 \sin^2\left(\frac{\theta}{2}\right) }.$$ 

    \begin{figure}

    \centering
    \begin{tikzpicture} [scale = 0.6]
\clip (-9.5,-2) rectangle (14,10);
  \def\g{9cm}
  \coordinate (O) at (0, 0);

  \draw[thin,name path=W] (0,0) circle (\g);

  \foreach \k in {1,...,5} {
    \pgfmathsetmacro{\angle}{90 + 72*(\k-1)}
    \coordinate (P\k) at (\angle:\g);
    \fill (P\k) circle (1pt);
    \node at (\angle:\g + 0.3cm) {$P_{\k}$};
  }

  \foreach \k in {1,...,5} {
    \pgfmathsetmacro{\angle}{75 + 72*(\k-1)}
    \coordinate (Q\k) at (\angle:\g);
    \fill (Q\k) circle (1pt);
    \node at (\angle:\g + 0.3cm) {$Q_{\k}$};
  }

  \fill (0,0) circle (1pt);
  \node[below left=0pt] at (0,0) {$O$};

  \draw[thin, dotted] (0,0) -- (P1) node[ midway, left] {$g$};
  \draw[thin, dotted] (0,0) -- (Q1);
  \draw[thin, dotted] (0,0) -- (P5);
  \draw[thin, dotted] (0,0) -- (Q5);
  \draw[thin, dotted] (P1) -- (Q1);
  \draw[thin, dotted] (P5) -- (Q5);
  \draw[thin, dotted] (Q1) -- (P5);

  \draw[dotted, name path=cQ1] (Q1) circle ({\g/2});
  \draw[dotted, name path=cP5] (P5) circle ({\g/2});
  \draw[dotted, name path=cQ5] (Q5) circle ({\g/2});
  \draw[dotted] (Q2) circle ({\g/2});
  \draw[dotted] (P2) circle ({\g/2});
  \draw[dotted] (P1) circle ({\g/2});
  \draw[dotted] (Q3) circle ({\g/2});
  \draw[dotted] (P4) circle ({\g/2});

  \path [name intersections={of=cQ1 and cP5, by={AA, A}}];
  \path [name intersections={of=cQ5 and cP5, by={BB, B}}];

  \fill[name path=cQ] (A) circle (1pt);
  \fill[name path=cQ] (B) circle (1pt);

  \draw[red] (0,0) -- (A) node[ sloped, font=\tiny, midway, above] {$a$}-- (B) -- (0,0) node[sloped, font=\tiny, midway, below] {$b$};

  \pic [draw=black, angle radius=1cm, font=\tiny, "$\frac{\pi}{n}$"] {angle = B--O--A};
  \pic [ draw = black, angle radius=1.3cm, font=\tiny, "$\theta$"] {angle = Q1--O--P1};

  \coordinate (M1) at ($(Q1)!0.5!(P5)$);
  \coordinate (M2) at ($(Q5)!0.5!(P5)$);

  \draw[blue] (Q1) -- (A) node[sloped, font=\tiny, midway, below] {$\epsilon$} -- (M1) -- (Q1) node[ sloped, font=\tiny, midway, above] {$l_2$};

  \draw[blue,] (P5) -- (B) node[sloped, font=\tiny, midway, above] {$\epsilon$} -- (M2) -- (P5) node[sloped, font=\tiny, midway, above] {$l_1$};

  \draw [decorate, decoration={brace, amplitude=4mm, mirror, raise=2pt}]
        (O) -- (M2) node [midway, sloped, below=3mm, font=\small] {$h_1$};

  \draw [decorate, decoration={brace, amplitude=4mm, raise=2pt}]
        (O) -- (M1) node [midway, sloped, above=4mm, font=\small] {$h_2$};

 \fill (A) circle (1pt);
  \node[below right=0pt] at (A) {$E_1$};
   \fill (B) circle (1pt);
  \node[above left=0pt] at (B) {$E_2$};
\end{tikzpicture}
    \vspace{0.3cm}
    \caption{Schematic setup for the calculation of $A_T(\epsilon)$.
    } 
    \label{fig:triangle}
\end{figure}

    The area of the circular segment (red in Figure \ref{fig:segment}) is in the notation of  Figure \ref{fig:segment} given by 
     $$A_S(\epsilon)=\frac{\epsilon^2}{2}\cdot[ \gamma(\epsilon) - \sin(\gamma(\epsilon))].$$
     Where the angle $\gamma(\epsilon)$ is determined by 
      $$\gamma (\epsilon) = \tilde{\gamma} - \tilde{\phi}(\epsilon) + \phi(\epsilon),$$
      with 
    $$\tilde{\gamma}=  \frac{\pi}{2} - \alpha = \frac{\pi}{2} - \frac{\pi}{n} + \frac{\theta}{2}$$
    and
   $$\tilde{\phi}(\epsilon) = \arccos \left(\frac{l_2}{ \epsilon}\right).$$
   
   By the law of sines it is
   $$\frac{\epsilon}{\sin(\theta/2)}=\frac{b(\epsilon)}{\sin(\phi(\epsilon))}$$
   and therefore 
   $$\phi(\epsilon)= \arcsin\left(\frac{b(\epsilon)}{\epsilon} \sin\left(\frac{\theta}{2}\right)\right).$$ 
   

    \begin{figure}

    \centering
    \begin{tikzpicture} [scale=1]
\clip (-2,-1) rectangle (10,6);
  \def\g{9cm}
  \coordinate (O) at (0, 0);

  \draw[thin,name path=W] (0,0) circle (\g);

  \foreach \k in {1,...,5} {
    \pgfmathsetmacro{\angle}{90 + 72*(\k-1)}
    \coordinate (P\k) at (\angle:\g);
    \fill (P\k) circle (1pt);
    \node at (\angle:\g + 0.3cm) {$P_{\k}$};
  }

  \foreach \k in {1,...,5} {
    \pgfmathsetmacro{\angle}{75 + 72*(\k-1)}
    \coordinate (Q\k) at (\angle:\g);
    \fill (Q\k) circle (1pt);
    \node at (\angle:\g + 0.3cm) {$Q_{\k}$};
  }

  \fill (0,0) circle (1pt);
  \node[below left=0pt] at (0,0) {$O$};

  \draw[thin, dotted] (0,0) -- (P1) node[ midway, left] {$g$};
  \draw[thin, dotted] (0,0) -- (Q1);
  \draw[thin, dotted] (0,0) -- (P5);
  \draw[thin, dotted] (0,0) -- (Q5);
  \draw[thin, dotted] (P1) -- (Q1);
  \draw[thin, dotted] (P5) -- (Q5);
  \draw[thin, dotted] (Q1) -- (P5);

  \path[dotted, name path=cP1] (P1) circle ({\g/2});
  \draw[dotted, name path=cQ1] (Q1) circle ({\g/2});
  \draw[dotted, name path=cP5] (P5) circle ({\g/2});
  \draw[dotted, name path=cQ5] (Q5) circle ({\g/2});

  \path [name intersections={of=cQ1 and cP5, by={AA, A}}];
  \path [name intersections={of=cQ5 and cP5, by={BB, B}}];

  \fill[name path=cQ] (A) circle (1pt);
  \fill[name path=cQ] (B) circle (1pt);

  \coordinate (M1) at ($(Q1)!0.5!(P5)$);
  \coordinate (M2) at ($(Q5)!0.5!(P5)$);

  \draw[thin, green] (O) -- (P5) -- (M1) -- (O);
  \draw[thin] (A) -- (B) -- (P5) -- (A);
  \draw[ red] (A) -- (B);
  \draw[red] (A) arc (153.93:205.7:\g/2);
  
  \draw[dotted] (O) -- (M2);
  \pic [ draw = red, angle radius=0.75cm, "$\color{red} \gamma$"] {angle = A--P5--B};
  \pic [ draw = blue, angle radius=1.5cm, "$\color{blue} \tilde{\gamma}$"] {angle = M1--P5--O};
  
\pic [
    draw = orange,
    angle radius = 1.3cm,
    angle eccentricity = 1.0,
    "" {name=A_pos} 
] {angle = M1--P5--A};

\draw[<-, orange, thin, >=stealth] (A_pos)+(0.1,-0.1) -- ++(0.7, 0.7) 
    node[above right] {$\tilde{\phi}$};

\pic [
    draw = orange,
    angle radius = 2cm,
    angle eccentricity = 1.0,
    "" {name=B_pos} 
] {angle = O--P5--B};

\pic [
    draw = black,
    angle radius = 2.3cm,
    angle eccentricity = 1.0,
    "" {name=C_pos} 
] {angle = B--O--P5};

\draw[<-, orange, thin, >=stealth] (B_pos)+(0.1,0.1) -- ++(0.7, -0.7) 
    node[below right] {$\phi$};

\draw[<-, thin, >=stealth] (C_pos)+(-0.2,-0.0) -- ++(0.7, -0.7) 
    node[below right] {$\frac{\theta}{2}$};

  \pic [draw=black, angle radius=0.7cm, font=\tiny, "$\alpha$"] {angle = P5--O--A};
  \pic [ draw = black, angle radius=1cm, font=\tiny, "$\theta$"] {angle = Q1--O--P1};
  \pic [ draw = black, angle radius=0.7cm, font=\tiny, "$\pi /2$"] {angle = O--M1--P5};

   \fill (A) circle (1pt);
  \node[above left=0pt] at (A) {$E_1$};
   \fill (B) circle (1pt);
  \node[below left=0pt] at (B) {$E_2$};
\end{tikzpicture}
    \vspace{0.3cm}
    \caption{Schematic setup for the calculation of $A_S(\epsilon)$.
    } 
    \label{fig:segment}
\end{figure} 

Calculate the second left derivative of $A_T(\epsilon) = \frac{1}{2} \cdot a(\epsilon) \cdot b(\epsilon) \cdot \sin(\frac{\pi}{n})$ in $\gg$:


On $(l_2, \infty)$ it is:

\[
A_T''
= \frac{1}{2}\sin\left(\frac{\pi}{n}\right)
\big(a''b+2a'b'+ab''\big).
\]

At $\epsilon=\gg$, since $a(\gg)=b(\gg)=0$ and $a''$, $b''$ are locally bounded at $\epsilon = \gg$,
\[
A_T''(\gg)
= \sin\left(\frac{\pi}{n}\right)a'(\gg)b'(\gg).
\]

Using
\[
a'(\gg)=-\sec\Big(\frac{\pi}{n}-\frac{\theta}{2}\Big),
\qquad
b'(\gg)=-\sec\frac{\theta}{2},
\]
we obtain
\begin{equation}
\label{eq:T''(g)}
    A_T''(\gg)
=\sin\Big(\frac{\pi}{n}\Big)
\sec\Big(\frac{\pi}{n}-\frac{\theta}{2}\Big)
\sec\frac{\theta}{2} \not = 0.
\end{equation}

Calculate the second left derivative of $A_S(\epsilon)=\frac{\epsilon^2}{2}\cdot[ \gamma(\epsilon) - \sin(\gamma(\epsilon))]$ in $\gg$:

With $\Gamma(\epsilon):=\gamma(\epsilon)-\sin(\gamma(\epsilon))$ we have $A_S(\epsilon)
= \frac{\epsilon^{2}}{2}\Gamma (\epsilon)$ and therefore on $(l_2, \gg)$ it is:

\begin{equation}
\label{eq:S''}
    A_S''(\epsilon)
= \Gamma (\epsilon)+2\epsilon \Gamma'(\epsilon)+\frac{\epsilon^{2}}{2}\Gamma''(\epsilon).
\end{equation}

With 
$$\Gamma'(\epsilon)
= \gamma'(\epsilon)\big(1-\cos\gamma(\epsilon)\big)$$
and
$$\Gamma''(\epsilon)
= \gamma''(\epsilon)\big(1-\cos\gamma(\epsilon)\big)
+ \gamma'(\epsilon)^{2}\sin\gamma(\epsilon).$$

At \(\epsilon=\gg\), since $\gamma(g)=0$, we obtain
$$\Gamma(\gg)=\Gamma'(\gg)=\Gamma''(\gg)=0$$
and therefore
\begin{equation}
\label{eq:S''(g)}
    A_S''(\gg)
= \Gamma(\gg)+2\gg\Gamma'(\gg)+\frac{\gg^{2}}{2}\Gamma''(\gg)
= 0.
\end{equation}

The Equations \ref{eq:T''(g)} and \ref{eq:S''(g)} together gives that the second left derivative of $H(\epsilon)$ at $\epsilon = \gg$ is non-zero.
    
\end{proof}

\begin{remark}
\label{remark:LemmaHoleFunction}
  Following the computations from the proof of Lemma \ref{lemma:holefunction}, one can easily see that $H(\epsilon)$ is $C^1$ for all $\epsilon$. Moreover, $H$ is a smooth function on its entire domain excluding the point g.
\end{remark}

Let the minimal gap width between the first-generation components of $\mathcal{F}$ be denoted by $$l:=\dist (\mathcal{S}_1(\mathcal{F}), \mathcal{S}_2(\mathcal{F})).$$

\begin{proposition}
\label{Prop:periodic}
The function $f(\epsilon) = \frac{\lambda^2(\mathcal{F}_\epsilon)}{\epsilon^{2-D}}$ is log-periodic on $(0, r^{-1}\cdot l/2].$ 
\end{proposition}

\begin{proof}

Recall that, in general, for a self-similar set $F$ generated by an self-similar system $\mathcal{S}=\{ S_1, \ldots , S_N \}$, where $\rho_i$ is the scaling ratio of $S_i$ for $i= 1, \ldots , N$, the following holds for all $\epsilon \ge 0$: $$S_i(F_\epsilon) = (S_i(F))_{\rho_i \epsilon}.$$

Let $r$ be the contraction ratio of $\mathcal{S}_i$ for all $i=1, \ldots ,n$. Therefore it holds $$\lambda^2 (\mathcal{F}_{r\cdot \epsilon}) = n\cdot r^{2} \lambda^2(\mathcal{F}_{\epsilon}) \ \ \ \forall \epsilon \leq r^{-1}\cdot l/2.$$   Further it is $r^{-D}= r^{\frac{\ln n}{\ln r}}=n$ and together we have
    $$f(r\cdot \epsilon)= \frac{ \lambda^2(\mathcal{F}_{r\cdot \epsilon})}{ (r\cdot \epsilon)^{2-D}}= \frac{n\cdot r^{2} \lambda^2(\mathcal{F}_{\epsilon})}{r^{2} \cdot r^{-D} \epsilon^{2-D}} = \frac{\lambda^2(\mathcal{F}_\epsilon)}{\epsilon^{2-D}} = f(\epsilon) \ \ \forall \epsilon \leq r^{-1}\cdot l/2.$$
   Thus $f$ is log-periodic with scaling factor $r$. 

\end{proof}

\begin{theorem}
\label{thm:mainThm}

  Let $n\in \NN$, $n \geq  3$ and let $\mathcal{F}^{n,r}$ be an $n$-flake dust with $r< l/(2g)$. Then  $\mathcal{F}^{n,r}$ is not Minkowski measurable. 
\end{theorem}

\begin{remark}
    The requirement $r< l/(2g)$  (equivalently $g < r^{-1}\cdot l/2$) guarantees that $g$ is within a periodic cycle of the function $f$, see the proof of Proposition \ref{Prop:periodic}. Note that for $n=4$ one has $r<l/(2g)$ for all $r \in (0, 1/2).$ 
\end{remark}

\begin{proof}
Fix $n\in \NN$, $n \geq 3$ and fix $r \in \RR$ such that $r< l/(2g)$.
By Proposition \ref{Prop:periodic} $f(\epsilon)  = \frac{\lambda^2(\mathcal{F}_\epsilon)}{\epsilon^{2-D}}$ is log-periodic with scaling factor $r$. It follows, if $\lambda^2(\mathcal{F}_\epsilon)$ is not a smooth function, then $\mathcal{F}$ is not Minkowski measurable by Lemma \ref{lemma:quotient_ist_nicht_konst}.

\textbf{Step 1:} Claim. The set of points in $\mathcal{F}$ at distance $g$ from a hole point of $\mathcal{F}$ is finite.

Clearly $0< g < \infty$ holds.  Let 
$$\hol_g (\mathcal{F}):= \{ x \in \hol (\mathcal{F}) : d(x, \mathcal{F}) = g \}$$
and for $c\in \hol_g (\mathcal{F})$ denote 
$$\Pi_c := \Pi_c(\mathcal{F}) := \{ y \in \mathcal{F} : \|c-y\| = g \}.$$

Further let be $K:= \conv( \{ V_1, \ldots ,V_n \})$ the convex hull of $ \{ V_1, \ldots ,V_n \}$, where $V_1, \ldots , V_n$ are the vertices of the underlying  $n$-gon, see Section \ref{sec:n-flake}.  Because of $ g= \sup \{ \epsilon > 0 : \mathcal{F}_\varepsilon \ \text{has a hole}  \}$, there is only one element in $\hol_g (\mathcal{F}) = \{O \}$, the  centroid  of $\mathcal{F}$. Thus the set $\Pi_O$ has $n$ elements if a vertex of $\mathcal{S}_1 K$ has the closest distance to the centroid of $\mathcal{F}$ , else $2n$ elements (see Figure \ref{fig:Pi_c} ) In particular $\Pi_O$ is a finite set.

 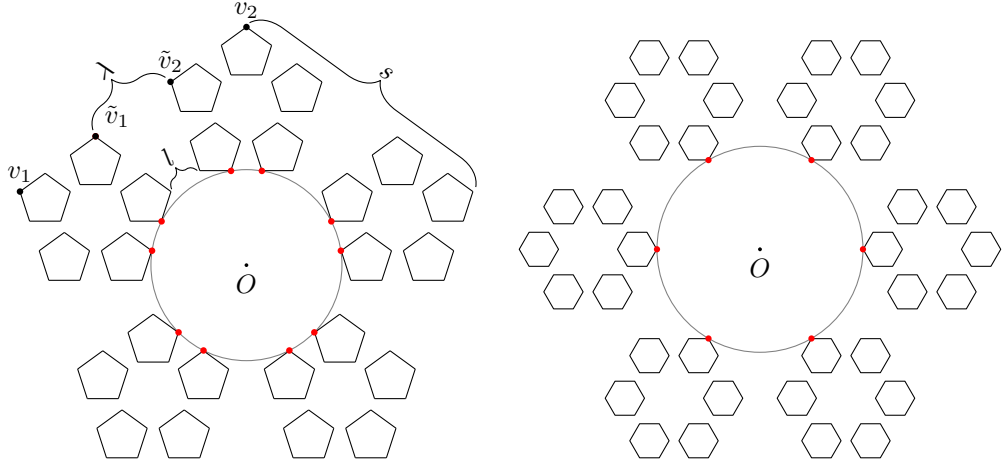
\begin{figure}
     \centering

\begin{tabular}{cc}

\begin{tikzpicture}[scale=0.35, line join=round]

  \def\R{9}      
  \def\rBlue{3} 
  \def\rRed{1} 

  \draw[thin, gray] (0,0) circle [radius=3.61];
  \fill (0,0) circle (2pt) node[below] {$O$};

  \foreach \a in {18, 90, 162, 234, 306} {
    
    \begin{scope}[shift={(\a:{\R-\rBlue})}]
      \path (18:\rBlue) -- (306:\rBlue) -- cycle; 

      \foreach \aa in {18, 90, 162, 234, 306} {
        
        \begin{scope}[shift={(\aa:{\rBlue-\rRed})}]
          
          \foreach \v in {18, 90, 162, 234, 306} {
            \draw[black, thin] (\v:\rRed) -- (\v+72:\rRed);
            
            \coordinate (P) at (\v:\rRed);
            
          }
          
        \end{scope}
      }
    \end{scope}
  }

\fill[red] (8.7:3.61) circle (3.5pt);
\fill[red] (27.3:3.61) circle (3.5pt);
\fill[red] (80.7:3.61) circle (3.5pt);
\fill[red] (99.3:3.61) circle (3.5pt);
\fill[red] (152.7:3.61) circle (3.5pt);
\fill[red] (171.3:3.61) circle (3.5pt);
\fill[red] (224.7:3.61) circle (3.5pt);
\fill[red] (243.3:3.61) circle (3.5pt);
\fill[red] (296.7:3.61) circle (3.5pt);
\fill[red] (315.3:3.61) circle (3.5pt);

\fill[red] (139.5:7.5) circle (3.5pt);

\fill[] (90:\R) circle (3.5pt) ;
\node[above]  at (90:\R) {$v_2$};
\fill[] (162:\R) circle (3.5pt);
\node[above]  at (162:\R) {$v_1$};
\fill[] (112.5:7.5) circle (3.5pt);
\node[above]  at (112.5:7.5) {$\tilde{v}_2$};
\fill[] (139.5:7.5) circle (3.5pt);
\node[above right]  at (139.5:7.5) {$\tilde{v}_1$};

\coordinate (A) at (90:9);
\coordinate (B) at (18:9);
\coordinate (C) at (67+72:7.53);
\coordinate (D) at (41+72:7.53);
\coordinate (E) at (63+72:3.83);
\coordinate (F) at (45+72:3.83);

\draw [decorate, decoration={brace, amplitude=4mm, raise=2pt}]
        (A) -- (B) node [midway, sloped, above=4mm] {$s$};
\draw [decorate, decoration={brace, amplitude=4mm, raise=2pt,}]
        (C) -- (D) node [midway, sloped, above=4mm] {$\lambda$};
\draw [decorate, decoration={brace, amplitude=2mm, raise=2pt,}]
        (E) -- (F) node [midway, sloped, above=2mm] {$l$};
\end{tikzpicture}

    & 

\begin{tikzpicture}[line join=round, scale=0.26]

  \def\R{12.25}       
  \def\rBlue{3.5}  
  \def\rRed{1}  

  \path[thick] (0:\R) -- (60:\R) -- (120:\R) -- (180:\R) -- (240:\R) -- (300:\R) -- cycle;

  \foreach \a in {0, 60, 120, 180, 240, 300} {
    
    \begin{scope}[shift={(\a:{\R-\rBlue})}]
      
      \path[blue, thick] (0:\rBlue) -- (60:\rBlue) -- (120:\rBlue) -- (180:\rBlue) -- (240:\rBlue) -- (300:\rBlue) -- cycle;

      \foreach \aa in {0, 60, 120, 180, 240, 300} {
        
        \begin{scope}[shift={(\aa:{\rBlue-\rRed})}]
          
          \draw[thin] (0:\rRed) -- (60:\rRed) -- (120:\rRed) -- (180:\rRed) -- (240:\rRed) -- (300:\rRed) -- cycle;
          
        \end{scope}
      }
    \end{scope}
  }
 \draw[thin, gray] (0,0) circle [radius=5.24];
   \fill (0,0) circle (2.8pt) node[below] {$O$};
 
 \fill[red] (0:5.24) circle (4.5pt);
 \fill[red] (60:5.24) circle (4.5pt);
 \fill[red] (120:5.24) circle (4.5pt);
 \fill[red] (180:5.24) circle (4.5pt);
 \fill[red] (240:5.24) circle (4.5pt);
 \fill[red] (300:5.24) circle (4.5pt);
 
\end{tikzpicture}
\\ 
\end{tabular}

     \caption{Elements of $\Pi_O$ (red) for $n = 5$ (left, 10 elements) and $n = 6$ (right, 6 elements).}
     \label{fig:Pi_c} 
    
 \end{figure} 

\textbf{Step 2:} Claim. The value $g \in \hv(\mathcal{F})$ is isolated.

By self-similarity of $\mathcal{F}$ it is enough to look at hole points $\tilde{c} \in K\setminus(\bigcup_{i=1}^n \mathcal{S}_iK)$. Note that then it holds $\tilde{c} \in \text{Unp}(\bigcup_{i=1}^n \mathcal{S}_iK)^{\mathsf{c}}\cap K$, where for $F \in \mathcal{H}(\mathbb{R}^d)$  $\mathrm{Unp}(F)$ denote the set of points in $\mathbb{R}^d$ having a unique nearest point in $F$. Observe that the centroid $O$ (by Step 1 the only element in $\hol_g (\mathcal{F})$) is the point in $\text{Unp}(\bigcup_{i=1}^n \mathcal{S}_iK)^{\mathsf{c}} \cap K$ having the largest distance to $\mathcal{F}$. If $n$ is odd, the centroid is also centered to the biggest gaps of $\big(\overline{K\setminus\bigcup_{i=1}^nS_i K}\big ) \cap \mathcal{F}$, therefore $g$ is isolated in $\hv (\mathcal{F})$ (see Figure \ref{fig:Pi_c}). 
If $n$ is even, vertices of $\mathcal{S}_1 K$ have the closest distance to the centroid of $\mathcal{F}$ and therefore $g$ is isolated. In both cases it is $g \leq 1 - 2r$, $s:= |\overline{V_1V_2|}=2 \sin(\pi/n)$ and $\lambda := |\overline{\tilde{V}_1 \tilde{V}_2}|= s - 2rs$, with $\tilde{V}_1:= V_1+r(V_2-V_1)$ and $\tilde{V}_2:= V_2+r(V_1-V_2)$, see Figure \ref{fig:Pi_c}. Therefore it is clear that $2g> \lambda$, thus $g$ is isolated.

 \begin{figure}
     \centering
     \begin{tikzpicture}[scale=0.5, line join=round]

  \def\R{6.25}      
  \def\rBlue{2.5}  
  \def\rRed{1}   

  \draw[gray] (90:\R) -- (210:\R) -- (330:\R) -- cycle;

  \foreach \a in {90, 210, 330} {
    
    \begin{scope}[shift={(\a:{\R-\rBlue})}]

      \filldraw[fill=blue!20, draw=black, line width=1pt] (90:\rBlue) -- (210:\rBlue) -- (330:\rBlue) -- cycle; 

      \foreach \aa in {90, 210, 330} {
        
        \begin{scope}[shift={(\aa:{\rBlue-\rRed})}]
          
          \draw[gray] (90:\rRed) -- (210:\rRed) -- (330:\rRed) -- cycle;
          
        \end{scope}
      }
       \draw[thick] (90:\rBlue) -- (210:\rBlue) -- (330:\rBlue) -- cycle; 
      
    \end{scope}
  }
    \draw[red, thick] (0,0) -- ($(90:\R) ! 0.5 ! (210:\R)$);
    \draw[red, thick] (0,0) -- ($(330:\R) ! 0.5 ! (210:\R)$);
    \draw[red, thick] (0,0) -- ($(90:\R) ! 0.5 ! (330:\R)$);
 \fill (0,0) circle (1.8pt) node[below left] {$O$};
\end{tikzpicture}
	
     \caption{The set $\text{Unp}(\bigcup_{i=1}^n S_iK)^{\mathsf{c}}\cap K$ (red) for $n=3$.}
     \label{fig:Ung} 
    
 \end{figure}
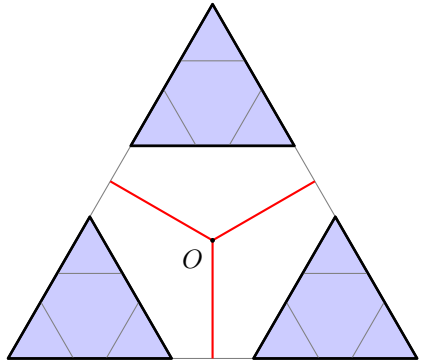 

\textbf{Step 3:} 
From Step 2 there exists a $\delta > 0$ such that $\mathcal{F}_\epsilon$ has exactly one hole for $\epsilon \in (g-\delta, g)$. From Step 1 we see that  for $\epsilon \in (g- \delta, g)$ the hole of $\mathcal{F}_\epsilon$ is the same as the hole of $(\Pi_O)_\epsilon$. But then the area of the hole is with Lemma \ref{lemma:holefunction} given by the function $H:(g-\delta, \infty) \to [0, \infty ), \epsilon \mapsto H(\epsilon)$ that has no second derivative in $g$.

\textbf{Step 4:} Claim. The function $\lambda^2(\mathcal{F}_\epsilon)$ is not smooth at $g$. 

Let $\mathcal{F}^*:=\mathcal{F} \cup B(O, g)$. By Step 3, for $\epsilon > g-\delta$ we have
\begin{equation}
\label{eq:LambdaOhneLoco}
    \lambda^2(\mathcal{F}_\epsilon)=\lambda^2(\mathcal{F}^*_\epsilon)- H(\epsilon).
\end{equation}

Therefore, we assume that $\lambda^2(\mathcal{F}^*_\epsilon) \in C^2((g-\delta, l)\setminus \{g\}).$ We will show that $\lambda^2(\mathcal{F}^*_\epsilon) \in C^2((g-\delta, l))$ holds: By Theorem \ref{Thm: Stracho}, we have 
$$\lambda^2(\mathcal{F}^*_\epsilon) \in C^2((g-\delta, l)\setminus \{g\}) \Leftrightarrow \mathcal{H}^1(\partial \mathcal{F}^*_\epsilon) \in C^1((g-\delta, l)\setminus \{g\})$$ 
and thus it is enough to show that $\mathcal{H}^1(\partial \mathcal{F}^*_\epsilon) \in C^1((g-\delta, l)).$ By self-similarity of $\mathcal{F}$ it holds that 
\begin{equation}
\label{eq: recursion}  
\mathcal{H}^1(\partial \mathcal{F}^*_\epsilon) = r \cdot 2 \cdot \big[\mathcal{H}^1(\partial \mathcal{F}^*_{r^{-1}\epsilon}) -r^{-1}\cdot 2 \epsilon \pi \big] +2 \epsilon \pi + R(\epsilon) \ \forall \epsilon \in (g-\delta, l/(2r)).
\end{equation}

By assumption, the first summand in Equation \eqref{eq: recursion} is $C^1$. The third summand $R(\epsilon)$ is given by 
$$R(\epsilon)= 2n\cdot  \epsilon \cdot \alpha(\epsilon) \ \forall \epsilon \in (g-\delta, l/(2r)),$$
where $\alpha(\epsilon) := \pi /2 -\beta (\epsilon)$ with $\beta (\epsilon ) := \angle \tilde{V}_1 \tilde{V}_2 w$.  Here, $w$ is the intersection point of the perpendicular bisector 
of $\overline{\tilde{V}_1 \tilde{V}_2}$ and $\partial \mathcal{F}_\epsilon$ that is closest to $\overline{\tilde{V}_1 \tilde{V}_2}$ (see Figure \ref{fig:R(epsilon)}). Clearly,  
$$\beta (\epsilon) = \arccos \left(\frac{|\overline{\tilde{V}_1 \tilde{V}_2}|}{2\epsilon} \right) \ \forall \epsilon \in (g-\delta, l/(2r)).$$
Thus, $R(\epsilon)$ is a $ C^1((g-\delta, l))$ function, which implies that $\lambda^2(\mathcal{F}^*_\epsilon) \in C^2((g-\delta, l))$. Consequently, by Equation \eqref{eq:LambdaOhneLoco}, Lemma \ref{lemma:holefunction} and Remark \ref{remark:LemmaHoleFunction}, $\lambda^2(\mathcal{F}_\epsilon)$ is not smooth at $g$. By Proposition \ref{Prop:periodic}, this concludes that $\mathcal{F}$ is not Minkowski measurable.

 \begin{figure}
     \centering

   \includegraphics[scale=0.8]{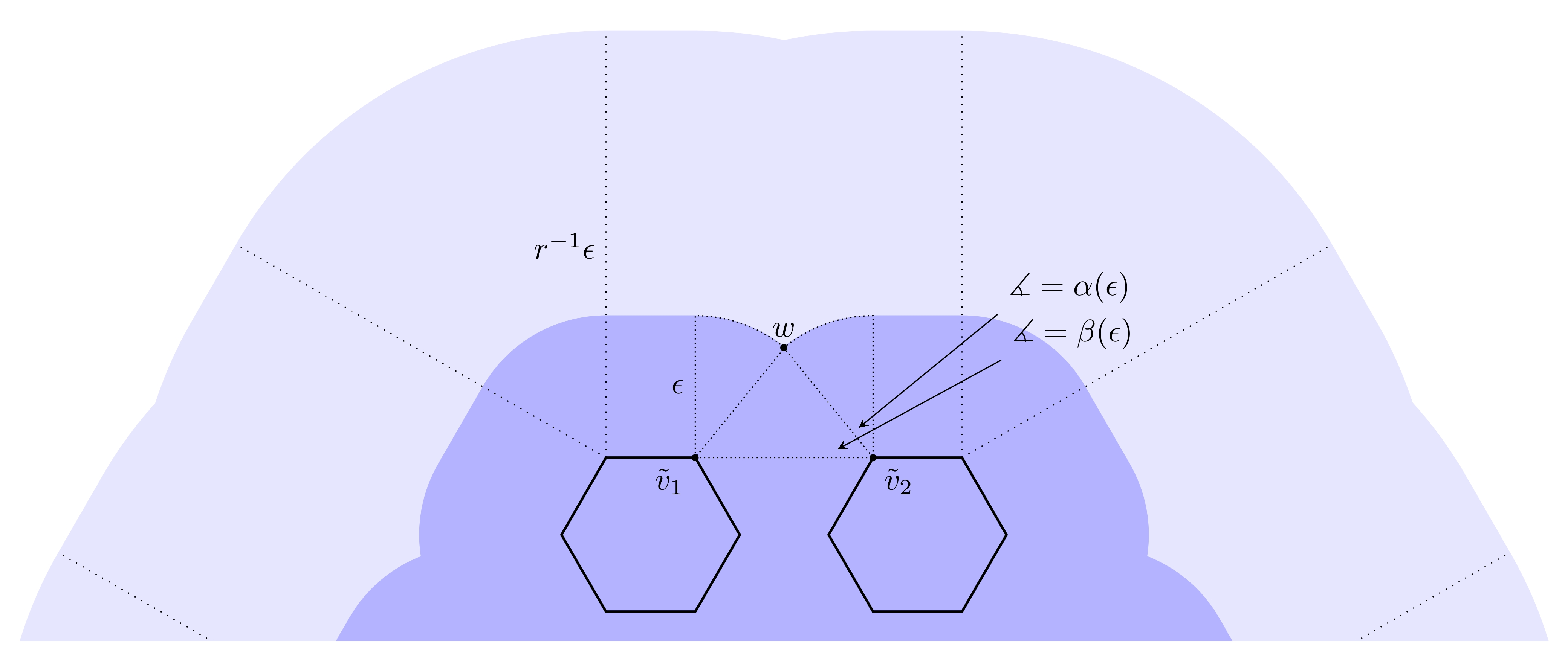}

     \caption{Schematic setup for the calculation of $R(\epsilon)$.}
     \label{fig:R(epsilon)} 
    
 \end{figure}

\end{proof}

\begin{remark}
For $n=4$ and $r<1/2$, we have $F^{n,r}= C^{1/r}$, the Cantor dust set discussed in \cite{freiberg2026family}.  The main result of \cite{freiberg2026family} states that $C^{1/r}$ is not Minkowski measurable for $r \leq 1/30$. For $1/30 < r < 1/2$, it is only conjectured (and supported by numerical evidence) that $C^{1/r}$ is not Minkowski measurable.
\end{remark}

\begin{remark}
    Note that if $r = R_n$, then the attractor is a \emph{Sierpinski n-gons} (for $n=3$ the classical Sierpinski Triangle). In this case the function $f(\epsilon) = \frac{\lambda^2(\mathcal{F}_\epsilon)}{\epsilon^{2-D}}$ is not periodic, cf. Proposition \ref{Prop:periodic}.
\end{remark}


 \section*{Acknowledgements} We would like to express our gratitude to Steffen Winter and Sebastian Lämmel for their careful proofreading. Their insightful comments greatly improved the clarity and accuracy of this work.  In particular, thanks are due to Steffen Winter for his idea to invoke Theorem \ref{Thm: Stracho}. Additionally, we acknowledge the support of the European Social Fund (ESF) and TU Chemnitz.

\bibliography{refs}

\end{document}